\documentclass[a4paper,12pt,english,francais]{amsart}

\usepackage{hyperref}

\usepackage{amssymb,calrsfs,lmodern,babel}\NoAutoSpaceBeforeFDP
\usepackage[T1]{fontenc}
\newcommand{\setC}{\mathbb{C}}

\newcommand{\setR}{\mathbb{R}}
\newcommand{\setZ}{\mathbb{Z}}

\newcommand{\bg}{\mathbf{g}}
\newcommand{\bR}{\mathbf{R}}
\newcommand{\cC}{\mathcal{C}}

\newcommand{\cL}{\mathcal{L}}
\newcommand{\cO}{\mathcal{O}}

\newcommand{\fkh}{\mathfrak{h}}
\newcommand{\fkg}{\mathfrak{g}}
\newcommand{\fM}{\mathfrak{M}}
\newcommand{\fR}{\mathfrak{R}}

\newcommand{\so}{\mathfrak{so}}

\renewcommand{\geq}{\geqslant}

\newcommand{\euc}{\mathrm{e}}
\newcommand{\etc}{\text{ etc.}}

\DeclareMathOperator{\Vol}{Vol}
\DeclareMathOperator{\End}{End}
\DeclareMathOperator{\Hom}{Hom}
\DeclareMathOperator{\Id}{Id}

\DeclareMathOperator{\Ric}{Ric}
\DeclareMathOperator{\Scal}{Scal}

\DeclareMathOperator{\Sym}{Sym}

\DeclareMathOperator{\Tr}{Tr}

\newtheorem{theo}{Th\'eor\`eme}[section]

\newtheorem{lemm}[theo]{Lemme}

\theoremstyle{definition}

\theoremstyle{remark}
\newtheorem{rema}[theo]{Remarque}
\newtheorem*{rema*}{Remarque}

\begin{document}

\author{Olivier Biquard}
\title{D\'esingularisation de m\'etriques d'Einstein. II.}
\address{UPMC Universit\'e Paris 6 et \'Ecole Normale Sup\'erieure, UMR 8553 du CNRS}
%\date{\today}
\thanks{L'auteur b\'en\'eficie du soutien du projet ANR-10-BLAN 0105.}

\selectlanguage{english}
\begin{abstract}
  We calculate a wall crossing formula for 4-dimensional
  Poincar\'e-Einstein metrics, through a wall made of orbifold
  Poincar\'e-Einstein metrics with $A_1$ singularities. This is based on
  a formalism which enables to deal with higher order terms of the
  Einstein equation in this setting. Some other consequences are
  deduced.
\end{abstract}

\maketitle

\selectlanguage{francais}

\section*{Introduction}
Cet article est le second d'une s\'erie commenc\'ee dans \cite{Biq13}, qui
vise \`a analyser les d\'esingularisations de m\'etriques d'Einstein
singuli\`eres en dimension 4. Si, dans le premier article, les outils
d'analyse \'etaient primordiaux, ici on se concentre sur les aspects
alg\'ebro-g\'eom\'etriques de l'\'equation d'Einstein, pour r\'epondre \`a
certaines questions soulev\'ees dans \cite{Biq13}. Mais rappelons
d'abord le contexte.

Soit $(M_0^4,g_0)$ une vari\'et\'e d'Einstein asymptotiquement
hyperbolique, c'est-\`a-dire que $M_0$ a un bord $\partial M_0=\{x=0\}$, o\`u $x$
est une \'equation du bord, et pr\`es de $\partial M_0$ on a 
$$ g_0 \sim \frac{dx^2+\gamma_0}{x^2}, $$
o\`u $\gamma_0$ est une m\'etrique sur $\partial M_0$ appel\'ee l'infini conforme de $g_0$
(en fait seule la classe conforme $[\gamma_0]$ est bien d\'efinie).

Le probl\`eme de Dirichlet \`a l'infini, consistant \`a essayer de
construire $g_0$ \`a partir de la donn\'ee de $\gamma_0$, a fait l'objet de
nombreux travaux, notamment d'Anderson \cite{And08} qui a eu l'id\'ee de la
construction d'un degr\'e comptant le nombre de solutions $g_0$ d'infini
conforme donn\'e. Dans \cite{Biq13}, on a montr\'e que, si $(M_0,g_0)$ a
une singularit\'e orbifold isol\'ee de type $\setR^4/\setZ_2$ en un point $p_0$,
et est non d\'eg\'en\'er\'ee, alors, sur la d\'esingularisation topologique $M$
(obtenue en \'eclatant $p_0$ en une sph\`ere d'auto-intersection $-2$),
une famille de d\'esingularisations d'Einstein de $(g_t)_{t\in (0,\varepsilon)}$
existe pourvu que $g_0$ satisfasse la condition
\begin{equation}
 \det \bR_+(g_0)(p_0) = 0.\label{eq:35}
\end{equation}
Ici $\bR_+$ est la partie autoduale de l'op\'erateur de courbure sur le
fibr\'e $\Omega^+$ des formes autoduales.

Si l'on varie un peu la m\'etrique conforme $[\gamma_0]$ en $[\gamma]$, alors on peut
d\'eformer $g_0$ en une m\'etrique d'Einstein $g_0(\gamma)$ sur $M_0$,
d'infini conforme $\gamma$. Notons $\cC$ l'espace de toutes les m\'etriques
conformes sur $\partial M_0$ et
\begin{equation}
  \label{eq:36}
  \cC_0 = \{ [\gamma]\in \cC, g_0(\gamma) \text{ satisfait (\ref{eq:35})} \}.
\end{equation}
On montre dans \cite{Biq13} que, si $\bR_+(g_0)(p_0)$ est de rang
\'egal \`a 2 \textemdash{} la r\'egularit\'e maximale possible compte tenu de
(\ref{eq:35}) \textemdash, alors $\cC_0$ est une hypersurface lisse de $\cC$ pr\`es
de $\gamma_0$, et que les m\'etriques d'Einstein d\'esingularis\'ees $g_t$ ont
leur infini conforme d'un seul c\^ot\'e de $\cC_0$.

Cela montre que la th\'eorie du degr\'e pour le probl\`eme AH Einstein ne peut
s'\'etendre au cas de vari\'et\'es qui, comme $M$, ont des sph\`eres
d'auto-intersection $-2$. Au contraire, on a un ph\'enom\`ene de \guillemotleft~mur~\guillemotright{} quand
l'infini conforme traverse l'hypersurface $\cC_0$, puisqu'une m\'etrique
d'Einstein dispara\^\i t, ou au contraire appara\^\i t. Si l'on pouvait construire la
th\'eorie ad\'equate de degr\'e dans cette situation, le degr\'e varierait de $+1$
ou $-1$ en traversant le mur. Pour le calculer, il est donc important de
d\'eterminer de quel c\^ot\'e du mur $\cC_0$ une m\'etrique appara\^\i t, et tel est
l'objet de cet article :
\begin{theo}\label{th1}
  Si $g_0$ est non d\'eg\'en\'er\'ee, et $\bR_+(g_0)(p_0)$ est de rang 2, alors
  les m\'etriques d'Einstein construites \`a partir de $\cC_0$ sont du
  c\^ot\'e de $\cC_0$ d\'etermin\'e par la condition
  $$ \det \bR_+(g_0(\gamma))(p_0) > 0. $$
\end{theo}
Cela apporte la r\'eponse \`a une question pos\'ee dans \cite[\S{} 14]{Biq13}
au vu du cas de la m\'etrique AdS-Taub-Bolt, qui est topologiquement le
quotient de la boule unit\'e de $\setR^4$ par $\setZ_2$ ; la m\'etrique ronde sur
le bord, infini conforme de la m\'etrique hyperbolique, est du c\^ot\'e $\det
\bR_+(g_0(\gamma))(p_0) < 0$.

La d\'emonstration du th\'eor\`eme \ref{th1} exige un calcul explicite des
termes d'ordre 2 de l'\'equation d'Einstein quand on recolle un
instanton gravitationnel \`a $M_0$. Ce calcul est bas\'e sur la g\'eom\'etrie
hyperk\"ahl\'erienne des instantons, qui s'\'etend \`a d'autres singularit\'es
que $\setC^2/\setZ_2$. Nous menons donc ces calculs dans le cadre plus g\'en\'eral
d'instantons gravitationnels, en g\'en\'eral orbifolds, de rang 1
(c'est-\`a-dire v\'erifiant $b_2^{orb}=1$), et le th\'eor\`eme final 
(th\'eor\`eme \ref{th:detRpositif}) est d\'emontr\'e dans un cadre plus
g\'en\'eral que le th\'eor\`eme \ref{th1}.

Ce formalisme plus g\'en\'eral a \'egalement l'avantage d'amorcer un proc\'ed\'e
de d\'esingularisation de toutes les singularit\'es kleiniennes : en
effet, dans le cas g\'en\'eral, l'instanton recoll\'e est certes encore
singulier, mais avec des singularit\'es orbifolds de rang strictement
inf\'erieur au rang initial, ce qui permet de commencer une induction
sur le rang. En particulier, on d\'etermine dans la section
\ref{sec:desing-des-sing} la seconde obstruction \`a la
d\'esingularisation d'une singularit\'e de type $A_k$ (c'est-\`a-dire
$\setC^2/\setZ_{k+1}$), qui est une fonction du 4-jet de la m\'etrique $g_0$ au
point $p_0$, et qui permet d'assurer que la m\'etrique $g_t$ construite
satisfasse encore l'obstruction (\ref{eq:35}) aux points singuliers
r\'esiduels.  Malheureusement, nous ne pouvons pas ici conclure en
appliquant directement les r\'esultats de \cite{Biq13}, car les
m\'etriques $g_t$ n'ont aucune raison d'\^etre non d\'eg\'en\'er\'ees (m\^eme si on
peut s'y attendre g\'en\'eriquement). Une analyse plus sophistiqu\'ee, bas\'ee
sur la seule non d\'eg\'en\'erescence de $g_0$, est n\'ecessaire, et sera
b\^atie dans le III de cette s\'erie d'articles. Disons juste ici que,
dans le cas de la singularit\'e $A_k$, on peut s'attendre \`a ce que la
d\'esingularisation soit possible pourvu que s'annulent $k$ obstructions
portant respectivement sur les 2-jet, 4-jet, \ldots, $2k$-jet de la
m\'etrique $g_0$ en $p_0$.

L'article est organis\'e de la mani\`ere suivante. Dans la premi\`ere section,
nous utilisons les liens entre tenseur de Ricci et connexion sur le fibr\'e
des 2-formes autoduales pour calculer les termes d'ordre 2 de l'\'equation
d'Einstein. Apr\`es en avoir donn\'e quelques propri\'et\'es dans la deuxi\`eme
section, nous appliquons aux instantons gravitationnels de rang 1 ce
formalisme dans la troisi\`eme section. Nous obtenons ainsi les
d\'eveloppements que l'on peut raccrocher au germe de $g_0$ en $p_0$ pour
r\'esoudre le probl\`eme d'Einstein de mani\`ere pr\'ecise dans la quatri\`eme
section et pour montrer le th\'eor\`eme \ref{th1}, non seulement pour la
singularit\'e $A_1$, mais aussi pour les singularit\'es $D_k$ et $E_k$. Pour
les singularit\'es $A_k$, d'autres termes apparaissent dans le d\'eveloppement,
qui demandent une \'etude plus pr\'ecise effectu\'ee dans la derni\`ere section, o\`u
on d\'etermine notamment l'obstruction \`a la d\'esingularisation provenant des
termes d'ordre 4 de la m\'etrique au point singulier (cette obstruction est
tu\'ee par l'invariance sous le groupe de la singularit\'e dans les autres
cas).

\section{\'Equation d'Einstein et 2-formes}
\label{sec:equat-deinst-et}

Dans cette section, nous installons un formalisme qui permet le calcul,
sous forme exploitable, des termes d'ordre 2 de l'\'equation d'Einstein. Ce
type de formalisme est pr\'esent dans la litt\'erature pour l'\'equation des
m\'etriques autoduales et d'Einstein, mais ne semble pas utilis\'e pour la
seule condition d'Einstein. Dans notre cadre, son int\'er\^et est de faire
appara\^\i tre naturellement le terme crucial qui permet le calcul.

Soit $M^4$ une vari\'et\'e et $V\to M$ un $SO_3$-fibr\'e vectoriel sur $M$, auquel
on pense comme le fibr\'e $\Omega^+_\bg\subset \Omega^2$ des 2-formes autoduales pour une
m\'etrique initiale $\bg$ sur $M$. Pour une autre m\'etrique $g$ sur $M$, on
peut trouver un morphisme de fibr\'es $\Phi:V\to\Omega^2$, tel que l'image de $\Phi$ soit
$\Omega^+_g$, et $\Phi$ soit une isom\'etrie.

\'Etant donn\'e $(V,\Phi)$, une connexion $\nabla$ sur $V$ est dite sans torsion si
pour toute section $v$ de $V$, on a $d(\Phi(v))=\Phi\land \nabla v$. (Ici, $\nabla v\in \Omega^1\otimes V$ et
$\Phi\in V^*\otimes\Omega^2$, donc le produit ext\'erieur $(V^*\otimes\Omega^2)\otimes(\Omega^1\otimes V)\to\Omega^3$ est la
contraction du facteur $V^*\otimes V$, suivie du produit ext\'erieur des formes). Le
lemme suivant est classique \cite[Proposition 2.3]{Fin11} :
\begin{lemm}
  La connexion de Levi-Civita de $g$, tir\'ee en arri\`ere par $\Phi$ sur $V$, est
  l'unique connexion m\'etrique et sans torsion sur $V$.
\end{lemm}
Ce lemme permet le calcul effectif de la connexion \`a partir de $\Phi$ de la
mani\`ere suivante. Soit $(v_1,v_2,v_3)$ une base orthonormale orient\'ee de
$V$, au sens o\`u $\langle v_i,v_j\rangle=2\delta_{ij}$ (convention choisie pour co\"\i ncider avec
la base naturelle de $\Omega^+(\setR^4)$ donn\'ee par $dx^1\land dx^2+dx^3\land dx^4$, $dx^1\land
dx^3-dx^2\land dx^4$ et $dx^1\land dx^4+dx^2\land dx^3$), alors une $SO_3$-connexion
sur $V$ est donn\'ee par une 1-forme de connexion
$$ a = \sum_1^3 a_i \otimes v_i, $$
o\`u $a_i\in \Omega^1$ et $v_i\in V\simeq \so(V)$ ;ici, un \'el\'ement $v\in V$ est identifi\'e \`a un
\'el\'ement de $\so(V)$ par produit vectoriel par $v$.

Par le lemme, notant $\Phi_i=\Phi(v_i)$, la connexion de Levi-Civita $a$ pour
$(V,\Phi)$ est caract\'eris\'ee par les \'equations :
$$ d\Phi_1 = a_3 \land \Phi_2 - a_2 \land \Phi_3, \etc $$
o\`u \guillemotleft{} etc. \guillemotright{} signifie que la m\^eme formule reste vraie par permutation
circulaire des indices. Notons $J_i$ la structure presque-complexe
associ\'ee \`a $\Phi_i$ par $\Phi_i(\cdot,\cdot)=g(J_i\cdot,\cdot)$, de sorte que $J_1J_2=J_3$,
etc. Alors, utilisant la formule $* \Phi_i\land\cdot = J_i\cdot $ sur les 1-formes, on
d\'eduit
$$ J_1 * d\Phi_1 = J_2 a_2 + J_3 a_3, \etc $$
qui se r\'esout en
$$ J_1a_1 = \frac12 \big( - J_1 * d\Phi_1 + J_2 * d\Phi_2 + J_3 * d\Phi_3 \big), \etc$$
qu'on peut \'ecrire aussi
\begin{equation}\label{eq:1}
a_1 = \frac12 \big( d^*\Phi_1 + J_3 d^* \Phi_2 - J_2 d^* \Phi_3 \big), \etc
\end{equation}
La courbure de $(V,\Phi)$, \`a savoir $R=da+\frac12[a,a]\in \Omega^2\otimes V$, se
d\'ecompose en parties autoduale et anti-autoduale, $R = R_+ + R_-$. 
Ici notre convention est que $R$ est bien la courbure, au sens usuel, du
fibr\'e $\Omega^+$, et donc $R=-\bR$, o\`u $\bR$ est l'op\'erateur de courbure
riemannien. En particulier, on a les formules usuelles liant $R$ \`a la
partie sans trace $\Ric_0$ du tenseur de Ricci et au tenseur de Weyl $W$ :
\begin{equation}
 R_+ = -\big( \frac{\Scal}{12} + W_+\big), \quad R_- = \Ric_0.\label{eq:37}
\end{equation}
Dans cette \'equation, l'identification de $R_-$ avec $\Ric_0$ se fait en
consid\'erant $R_-$ comme \'el\'ement de $\Hom(\Omega^-,\Omega^+)\simeq \Omega^-\otimes\Omega^+$ et $\Ric_0$
comme \'el\'ement de $\Sym_0^2 TM$, et en identifiant $\Omega^-\otimes\Omega^+\simeq \Sym_0^2 TM$
par
\begin{equation}
  \label{eq:38}
  u_-\otimes u_+\mapsto u_-\circ u_+=u_+\circ u_-,
\end{equation}
o\`u $u_\pm$ sont consid\'er\'es comme des endomorphismes anti-sym\'etriques de $TM$.

Appliquons cette technique de calcul aux petites d\'eformations d'une
m\'etrique hyperk\"ahl\'erienne $\bg$. Le fibr\'e $V=\Omega^+_\bg$, plat, est
trivialis\'e par les trois formes de K\"ahler $(\omega_1,\omega_2,\omega_3)$ de $\bg$. Une
autre m\'etrique $g$ est donn\'ee par un morphisme
$\Phi:V\to\Omega^2=\Omega^+_\bg\oplus\Omega^-_\bg$, que l'on peut choisir du type
\begin{equation}
 \Phi(v_i) = \exp\begin{pmatrix}\lambda&-\phi^t\\-\phi&\lambda\end{pmatrix}\omega_i\label{eq:2}
\end{equation}
pour un morphisme $\phi:\Omega^+_\bg\to\Omega^-_\bg$ et un r\'eel $\lambda$. La d\'eformation se
d\'ecompose en une d\'eformation conforme de $\bg$ en $e^\lambda \bg$, et une
d\'eformation de la structure conforme par $\exp(\begin{smallmatrix}0&-\phi^t\\-\phi&0\end{smallmatrix})$.

Notant $\phi_i=\phi(\omega_i)\in \Omega^-_\bg$, on a le terme de premier ordre
$$ \Phi_i^{(1)} = \lambda \omega_i - \phi_i , $$
li\'e au premier ordre de d\'eformation de la m\'etrique par
$$ g^{(1)} = \lambda + \sum_1^3 \omega_i\circ \phi_i, $$
o\`u l'on a utilis\'e l'identification $\Omega^-\otimes\Omega^+\simeq \Sym_0^2TM$ explicit\'ee dans
(\ref{eq:38}). Le calcul suivant jouera un r\^ole important :
\begin{lemm}\label{lem:Ric2}
  Supposons $(M,\bg)$ hyperk\"ahl\'erienne. Soit $\Phi$ donn\'e par
  (\ref{eq:2}), satisfaisant la condition de jauge $\sum_1^3 J_i * d\phi_i +
  d\lambda=0$. Alors, au premier ordre,
$$ a_i^{(1)} = * d\phi_i, \quad R_i^{(1)} = d * d \phi_i. $$

Si en outre la d\'eformation $\Phi$ est infinit\'esimalement d'Einstein, alors
$R_i^{(1)}\in \Omega^+_\bg$ est combinaison lin\'eaire des $\omega_j$,
$$ R_i^{(1)} = \sum_1^3R_{ij} \omega_j ; $$
les coefficients $R_{ij}$ sont des fonctions harmoniques r\'eelles, et les
termes d'ordre 2 de $\Ric_0$ sont donn\'es par
$$ \Ric_0^{(2)} = d_-a^{(2)} + \frac12 [a^{(1)},a^{(1)}]_- - \phi(R_+^{(1)}).$$
\end{lemm}
\begin{rema*}
  1\textdegree{} La condition de jauge est exactement la jauge de Bianchi
  $B_\bg g^{(1)}=\delta_\bg g^{(1)}+\frac12 d\Tr g^{(1)}=0$. On retrouve ainsi que la
  lin\'earisation de l'\'equation d'Einstein en jauge de Bianchi,
  $\Ric(g)+\delta^*_gB_\bg$, pour une m\'etrique hyperk\"ahl\'erienne $\bg$,
  s'\'ecrit
  \begin{equation}
   P_\bg = d_-d_-^*\label{eq:16}
  \end{equation}
  sur la partie sans trace de la m\'etrique (et $\frac12 \Delta$ sur la
  partie \`a trace).

  2\textdegree{} Puisque $R_i^{(1)}=da_i^{(1)}$, cette formule implique que la classe de
  cohomologie de $R_i^{(1)}$ soit triviale. Sur une vari\'et\'e compacte, c'est
  impossible, sauf si $R_i^{(1)}=0$ : on retrouve ainsi que, sur une vari\'et\'e
  compacte, une d\'eformation d'Einstein d'une m\'etrique hyperk\"ahl\'erienne
  reste hyperk\"ahl\'erienne.

  3\textdegree{} La param\'etrisation pr\'ecise des termes d'ordre sup\'erieur de la
  m\'etrique, \`a partir de $(\lambda,\phi)$, par (\ref{eq:2}), n'a pas d'importance
  particuli\`ere. Une param\'etrisation diff\'erente modifierait les termes
  d'ordre 2 de la m\'etrique, et donc le terme $d_-a^{(2)}$ dans les termes
  d'ordre $2$ de $\Ric_0$, qui ne jouera pas de r\^ole dans le calcul des
  obstructions. Plus loin, on utilisera aussi une param\'etrisation de la
  m\'etrique plus classique, par des 2-tenseurs sym\'etriques :
  $g=\bg+g^{(1)}+g^{(2)}+\cdots $. En faisant de la connexion induite sur $\Omega^+$
  l'objet g\'eom\'etrique central, le formalisme introduit dans cette section
  sert \`a \'ecrire facilement les termes quadratiques induits par les
  d\'eformations d'ordre 1.
\end{rema*}
\begin{proof}
  La formule sur $a_i^{(1)}$ et $R_i^{(1)}$ est le r\'esultat d'un calcul
  imm\'ediat. Si la d\'eformation infinit\'esimale est d'Einstein, alors
  $R_-^{(1)}=d_-a^{(1)}=0$, donc $R^{(1)}=R_+^{(1)}=d_+a^{(1)}$ satisfait
  $dR^{(1)}=0$. \'Ecrivons $R^{(1)}=\sum R_i^{(1)} v_i = \sum R_{ij}\omega_j\otimes v_i$, alors
  $dR_i^{(1)}=0$, donc, puisque $R_i^{(1)}$ est autoduale, $\Delta R_i^{(1)}=0$,
  ce qui force $\Delta R_{ij}=0$ pour chaque $(i,j)$.

  Enfin, $\Ric_0=\pi_{\Omega^-}(R)$, d'o\`u la derni\`ere formule, o\`u le terme
  $\phi(R_+^{(1)})$ tient compte de la variation de $\Omega^-$.
\end{proof}

\section{Les espaces ALE de rang 1}
\label{sec:les-espaces-ale}

\subsection{La construction de Kronheimer}
\label{sec:la-construction-de}

Soit $(Y^4,g)$ un orbifold ALE hyperk\"ahl\'erien, donc issu de la construction
de Kronheimer \cite{Kro89a}, que nous utilisons intensivement dans cette
section. \`A l'infini, $(Y,g)$ est asymptote \`a $\setC^2/\Gamma$, avec $\Gamma\subset SU_2$
fini. Nous disons que $Y$ est de rang 1 si
$$b_2^{orb}(Y)=1.$$
L'exemple de base est la m\'etrique de Eguchi-Hanson sur $T^*P_\setC^1$,
d\'esingularisation de $\setC^2/\setZ_2$, mais de tels exemples existent parmi
les d\'eformations de n'importe quelle singularit\'e fuchsienne $\setC^2/\Gamma$, comme
on va le voir maintenant.

Soit $\fkg$ l'alg\`ebre de Lie associ\'ee au groupe $\Gamma$ par la correspondance de
McKay : c'est donc une alg\`ebre de Lie simple de type A, D ou E. Soit $\fkh\subset
\fkg$ une sous-alg\`ebre de Cartan, et $\fR\subset \fkh^*$ un syst\`eme de
racines. Pour $\theta\in \fR_+$ soit l'hyperplan $D_\theta=\ker \theta$ de $\fkh$.  Alors les
instantons gravitationnels asymptotes \`a $\setC^2/\Gamma$ sont param\'etr\'es par les
triplets $\zeta=(\zeta_1,\zeta_2,\zeta_3)\in \fkh\otimes\setR^3$, et on notera $Y_\zeta$ l'espace
correspondant. Il est lisse si $\zeta\notin D_\theta\otimes\setR^3$ pour tout $\theta\in \fR_+$, et dans
ce cas $H^2(Y_\zeta,\setR)$ s'identifie \`a $\fkh$. Dans le cas o\`u $\zeta_2=\zeta_3=0$, alors
$Y_\zeta$ est la d\'esingularisation de $\setC^2/\Gamma$, dont le diviseur exceptionnel
consiste en une configuration de $k$ courbes holomorphes, formant une base
de $H_2(Y_\zeta,\setZ)$, et correspondant aux $k$ racines simples $\theta_1,\ldots,\theta_k$ de
$\fkh$. (La configuration des courbes donne le diagramme de Dynkin de $\fkg$).

Le cas de rang 1 est au contraire celui d'une plus petite
d\'esingularisation partielle, contenant une seule courbe holomorphe. On
en construit de la mani\`ere suivante : soit $i_0\in \{1,\ldots,k\}$ et $L\subset \fkh$
la droite d\'efinie par $L=\cap_{i=1,\ldots,k, i\neq i_0} D_{\theta_i}$ ; alors si $\zeta\in
L\otimes\setR^3$, l'instanton $Y_\zeta$ est un orbifold, dont le $H^2_{orb}$
s'identifie \`a $L$. L'indice $i_0$ correspond \`a un n\oe ud du diagramme de
Dynkin, et $Y_\zeta$ a alors un nombre de points singuliers \'egal au nombre
de composantes connexes du diagramme de Dynkin priv\'e du n\oe ud
correspondant \`a $i_0$. Chacune de ces composantes est encore un
diagramme de Dynkin, correspondant au sous-groupe fini $\Gamma'\subset SU_2$ tel
que le point singulier soit de type $\setC^2/\Gamma'$.

En particulier, si on choisit $i_0$ de sorte que le n\oe ud correspondant soit
\`a une extr\'emit\'e du diagramme de Dynkin, alors $Y_\zeta$ admet un seul point
orbifold, de type $\setC^2/\Gamma'$ o\`u $\Gamma'$ est un sous-groupe fini de $SU_2$ dont
le diagramme de Dynkin correspondant est de rang $k-1$. Ainsi obtient-on \`a
partir de $A_k$ une singularit\'e $A_{k-1}$, \`a partir de $D_k$ une
singularit\'e $A_{k-1}$ ou $D_{k-1}$, et \`a partir de $E_k$ ($k=6$, $7$, $8$)
une singularit\'e $A_{k-1}$, $D_{k-1}$ ou $E_{k-1}$ ($k\neq 6$).

Fixons \`a pr\'esent un instanton gravitationnel, $Y=(Y_\zeta,\bg)$, de rang
1. Compte tenu de l'action de $SO_3$ sur les $\zeta=(\zeta_1,\zeta_2,\zeta_3)$ et de
l'appartenance des $\zeta_i$ \`a une droite $L\subset \fkh$, on peut se ramener au
cas o\`u $\zeta_2=\zeta_3=0$, c'est-\`a-dire choisir la structure complexe $I_1$ de
sorte qu'elle soit la structure complexe correspondant \`a la
d\'esingularisation partielle de $\setC^2/\Gamma$, obtenue en ajoutant une seule
courbe holomorphe $\Sigma$.

D\'ecrivons bri\`evement la construction d'un tel instanton : si $\fM$
est la repr\'esentation r\'eguli\`ere du groupe fini $\Gamma$, alors $Y$ est 
un quotient hyperk\"ahl\'erien $(\End\fM\otimes\setC^2)^\Gamma /// PSU(\fM)^\Gamma$, o\`u $\Gamma$
agit sur $\setC^2$ par l'inclusion naturelle $\Gamma\subset SU_2$.  Les points du quotient
sont les solutions (modulo l'action de $PSU(\fM)^\Gamma$) des \'equations, pour
$(\alpha,\beta)\in (\End\fM\times\End\fM)^\Gamma$ :
\begin{align*}
  [\alpha,\alpha^*]+[\beta,\beta^*] &= \zeta_1, \\
  [\alpha,\beta] &= \zeta_2+i\zeta_3 = 0.
\end{align*}
La r\'esolution partielle des singularit\'es
$$ p:Y_{\zeta_1,0,0}\to Y_{0,0,0}=\setC^2/\Gamma $$
s'obtient en consid\'erant $(\alpha,\beta)$ dans le quotient symplectique complexe
$(\End\fM\otimes\setC^2)^\Gamma // PSL(\fM)^\Gamma$.

\subsection{L'action de cercle}
\label{sec:laction-de-cercle}

L'instanton gravitationnel admet une action holomorphe de $\setC^*$ pour
la structure complexe $J_1$, donn\'ee par $(\alpha,\beta)\mapsto(z\alpha,z\beta)$ (cela commute
bien \`a l'action de $\Gamma$). Comme le seul point fixe sur
$(\End\fM\otimes\setC^2)^\Gamma$ est l'origine, et comme l'action de $\setC^*$ commute \`a
la r\'esolution partielle $p$, les points fixes de l'action de $\setC^*$ sur
$Y$ doivent \^etre sur la courbe rationnelle $p^{-1}(0) = \Sigma$. Dans le
cas $A_1$, on a $Y=T^*\setC P^1$ et tous les points de $\Sigma$ sont fixes ;
dans les autres cas, l'action est g\'en\'eralement non triviale sur $\Sigma$ et
fixe les points singuliers restant sur $\Sigma$.

Restreignant l'action de $\ \setC^*$ au cercle, on obtient une action
isom\'etrique de $S^1$ sur $Y$, holomorphe pour $J_1$, agissant sur le couple
$(J_2,J_3)$ par rotation. Notons $\xi$ le champ de vecteurs sur $Y$ induit
par l'action infinit\'esimale de $S^1$. On a alors l'observation suivante :
\begin{lemm}[{\cite[proposition 5.1]{Hit98}}]\label{lem:ale-s1}
  Soit $m$ une application moment pour l'action de cercle sur $(Y,\omega_1)$,
  c'est-\`a-dire satisfaisant $dm=-\xi\lrcorner \omega_1$, alors $m$ est un
  potentiel pour $(\omega_2,J_2)$ et $(\omega_3,J_3)$, c'est-\`a-dire $\omega_i=\frac12 dJ_idm$ pour
  $i=2$, $3$.
\end{lemm}
\begin{proof}
  Pour fixer correctement les constantes, nous rappelons la preuve. La
  forme symplectique complexe $\omega_2+i\omega_3$ sur $\setC^2$ est $dz^1\land dz^2$, donc
  l'action de cercle est de poids 2, d'o\`u
  $\cL_{\xi}(\omega_2+i\omega_3)=2i(\omega_2+i\omega_3)$. Il en r\'esulte $\cL_{\xi}\omega_3=2\omega_2$. Or
  $ J_2 dm = -J_2 (\xi\lrcorner \omega_1) = \xi\lrcorner \omega_3 $, donc $$dJ_2dm =
  \cL_{\xi}\omega_3 = 2\omega_2.$$
\end{proof}
Il est facile de voir que 
\begin{equation}
m(\alpha,\beta)=\frac12 ( |\alpha|^2+|\beta|^2 ).\label{eq:3}
\end{equation}
On peut aussi montrer directement que le potentiel de $\omega_2$ et $\omega_3$ donne
le potentiel sur $Y$ parce que $\zeta_2=\zeta_3=0$, voir \cite[th\'eor\`eme 7]{BiqGau97}.

Dans la suite nous choisirons (\ref{eq:3}) pour le potentiel de $Y$ pour
$J_2$ et $J_3$, ce qui permet de lever l'ambigu\"\i t\'e sur la constante
additive. Ce choix se d\'ecrit de la mani\`ere suivante : les instantons
gravitationnels co\"\i ncident \`a l'ordre $4$ avec la m\'etrique euclidienne
\cite[proposition 3.14]{Kro89a} :
\begin{equation}
   \bg - \euc = O\big(\frac1{r^4}\big).\label{eq:4}
\end{equation}
(Il est facile de le v\'erifier via la description ci-avant par
quotient). Alors le choix (\ref{eq:3}) pour le potentiel est caract\'eris\'e
par le d\'eveloppement \`a l'infini :
\begin{equation}
  \label{eq:5}
  m = \frac{r^2}2 + O\big(\frac1{r^2}\big).
\end{equation}

\subsection{La cohomologie $L^2$}
\label{sec:la-cohomologie-l2}

Dans un espace ALE, la cohomologie $L^2$ en degr\'e 2 est \'egale \`a la
cohomologie \`a support compact. Il en r\'esulte imm\'ediatement que $Y$, de
rang 1, admet une unique forme harmonique $L^2$, duale de Poincar\'e de
$2\pi\Sigma$, que nous noterons $\Omega$ : donc $\Omega\in \Gamma(\Omega^-)$ et $d\Omega=0$.
\begin{lemm}\label{lem:ale-potentiel}
  La forme $\Omega$ admet un potentiel global $\psi$, singulier sur $\Sigma$ :
  $\Omega=dd^C\psi$ avec :
  \begin{itemize}
  \item pr\`es de $\Sigma$, $\psi \sim -\frac12 \ln |\sigma|^2$, o\`u $\sigma\in H^0(\cO(\Sigma))$ est une
    \'equation de $\Sigma$ ;
  \item \`a l'infini, dans le cas d'un groupe $\Gamma$ correspondant aux
    singularit\'es $A_1$, $D_k$ ou $E_k$,
    $$\psi = \frac{\pi\Vol \Sigma}{\Vol(S^3/\Gamma)}\,\frac1{r^2}
          + O\big(\frac1{r^6}\big),$$
    et dans le cas de la singularit\'e $A_k$ ($k\geq 1$),
    $$\psi = (k+1)\frac{\Vol \Sigma}{2\pi}\,\frac1{r^2}
    - (k^2-1)\left(\frac{\Vol \Sigma}{2\pi}\right)^2 \frac{|z^1|^2-|z^2|^2}{r^6}
    + O\big(\frac1{r^6}\big),$$
    o\`u $(z^1,z^2)$ sont des coordonn\'ees \`a
    l'infini diagonalisant l'action de $\Gamma=\setZ_{k+1}$.
  \end{itemize}
\end{lemm}
\begin{rema}\label{rema:11223344}
  Le facteur $|z^1|^2-|z^2|^2$ est intrins\`eque au sens suivant : les
  deux coordonn\'ees \`a l'infini $z^1$ et $z^2$ sont les directions
  propres de l'action de $\setZ_{k+1}$. En revanche, le signe r\'esulte d'un
  choix arbitraire d'\'ecriture de l'instanton gravitationnel ; il sera
  coh\'erent avec les autres calculs effectu\'es.
\end{rema}
\begin{proof}
  Soit $\chi$ une fonction \`a support compact, \'egale \`a $1$ dans un voisinage de
  $\Sigma$. Par la formule de Lelong-Poincar\'e, la classe de cohomologie de
  $\Omega+\frac12 dd^C(\chi\ln|\sigma|^2)$ est nulle, donc il existe une fonction $g\in
  L^2$ telle que $\Omega=dd^C(g-\frac12 \chi\ln|\sigma|^2)$. Posons $\psi=g-\frac12
  \chi\ln|\sigma|^2$ : puisque $\Omega$ est anti-autoduale, $\psi$ est harmonique.

  Compte tenu du comportement asymptotique (\ref{eq:4}), les premiers
  termes de $\psi$ \`a l'infini co\"\i ncident avec ceux d'une fonction harmonique
  sur $\setR^4$ :
  $$\psi=\frac a{r^2} + \frac{q(x)}{r^6} + O(\frac1{r^6}), $$
  o\`u $q$ est une forme quadratique sans trace sur $\setR^4$.
  La valeur du coefficient $a$ provient de l'int\'egration par parties
  suivante :
  \begin{equation*}
    0 = \int \Delta\psi \cdot 1 = -2\pi\Vol\Sigma - \lim_{r\to\infty} \int_{S_r} \frac{\partial\psi}{\partial r} 
                = -2\pi\Vol\Sigma + 2 a \Vol(S^3/\Gamma).
  \end{equation*}
  La forme quadratique sans trace $q(x)$ doit \^etre invariante sous
  l'action du groupe $\Gamma$ :
  \begin{itemize}\item 
    pour les groupes $\Gamma$ correspondant \`a $D_k$ et $E_k$ il n'existe pas
    de telle forme invariante, donc $q=0$ ; \item dans le cas $A_1$
    ($\Gamma=\setZ_2$), la m\'etrique $\bg$ est la m\'etrique de Eguchi-Hanson, de
    groupe d'isom\'etries $U_2$, donc le potentiel doit \^etre invariant
    sous l'action de $U_2$, ce qui impose aussi $q=0$. (On peut
    calculer explicitement ce potentiel).
  \end{itemize}
  Dans le cas d'une singularit\'e $A_k$, il y a 3 formes quadratiques
  invariantes sous $A_k$: $q_1=|z^1|^2-|z^2|^2$ et les parties r\'eelle et
  imaginaire de $z^1z^2=q_2+iq_3$, donc
  $$ q = a_1q_1+a_2q_2+a_3q_3.$$
  Pour calculer les coefficients $a_i$, soit $\varphi_i$ la fonction harmonique sur
  $Y$, quadratique \`a l'infini :
  $$ \varphi_i(x) = q_i(x) + O(r^{-2}). $$
  Alors la m\^eme int\'egration par parties que pr\'ec\'edemment donne
  \begin{align*}
  0 &= \int \Delta\psi \cdot \varphi_i \\
    &= -2\pi\int_\Sigma \varphi_i + \lim_{r\to\infty} \int_{S_r} \psi \partial_r\varphi_i - \partial_r \psi \varphi_i \\
    &= -2\pi\int_\Sigma \varphi_i + \lim_{r\to\infty} \int_{S_r} \frac{6q_i\varphi_i}{r^7},
  \end{align*}
  d'o\`u r\'esulte
  $$ a_i = \frac{\pi \int_\Sigma \varphi_i}{3 \int_{S^3/\Gamma} q_i^2}. $$
  \`A pr\'esent, rappelons que $Y$ est une r\'esolution partielle de
  $\setC^2/\Gamma=\setC^2/\setZ_{k+1}$, donc $z^1z^2$ est une fonction holomorphe globale
  sur $Y$, nulle sur $\Sigma$. Ainsi $\varphi_2+i\varphi_3=z^1z^2$ et donc $a_2=a_3=0$. La
  valeur de $a_1$ r\'esulte en revanche d'un calcul explicite avec l'ansatz
  de Gibbons-Hawking, laiss\'e en annexe \ref{sec:les-instantons-a_k}.
\end{proof}
On d\'eduit le terme principal \`a l'infini de $\Omega$ : \'ecrivant
$\psi=\frac{c_\Gamma}{r^2}+\cdots$, on calcule
\begin{equation}
 \Omega = \frac{4c_\Gamma}{r^4} (dr\land J_1dr-J_2dr\land J_3dr) +
 \begin{cases}
   O(r^{-8}) & \text{ pour }A_1, D_k, E_k,\\
   O(r^{-6}) & \text{ pour }A_k, k\geq 2.
 \end{cases}\label{eq:6}
\end{equation}
On trouve un premier terme asymptotique ind\'ependant de l'espace ALE
choisi (\`a constante pr\`es), et en particulier \'egal \`a celui de la
m\'etrique de Eguchi-Hanson utilis\'ee dans \cite{Biq13}.

Enfin, il y a un lien important entre la forme harmonique $\Omega$ et
l'application moment $m$. En effet, $\xi$ \'etant un vecteur de Killing,
dual (au sens de la m\'etrique) de la 1-forme $J_1dm$, on d\'eduit que
$$ \alpha=\nabla J_1dm $$ est en r\'ealit\'e une 2-forme, donc $\alpha=-\frac12 dd^Cm$. De $dJ_idm=2\omega_i$ pour
$i=2$, $3$, on d\'eduit imm\'ediatement la partie autoduale :
$$ \alpha^+ = - \omega_1. $$
Par cons\'equent, la partie anti-autoduale est ferm\'ee : $d\alpha^-=0$, c'est
donc une 2-forme harmonique, n\'ecessairement un multiple de $\Omega$,
d\'etermin\'e par la contrainte $\int_\Sigma\alpha=0$, donc :
$$ \alpha^- = s \Omega, \quad s = \frac{\int_\Sigma\omega_1}{\int_\Sigma\Omega}. $$

Par la th\'eorie g\'en\'erale dans les espaces \`a poids, on sait qu'existe
une 2-forme autoduale $\phi=O(r^2)$, telle que
$d_-d_-^*\phi=\Omega$. (La forme $\phi$ est unique \`a l'addition pr\`es d'un
multiple de $\Omega$). La condition $(d_-\oplus d^*)(d_-^*\phi)=\Omega\oplus0$
force l'\'egalit\'e $d^*\phi=-\frac 2s d^Cm$, c'est-\`a-dire
$$ J_1 d^*\phi - \frac 2s dm=0, $$
ce qu'on interpr\`ete en disant que $-\frac 2s m \bg + \omega_1\circ \phi$ est en
jauge de Bianchi.

\section{D\'eveloppements explicites sur l'instanton gravitationnel}
\label{sec:devel-expl-sur}

La r\'esolution du probl\`eme d'Einstein sur le recollement de $(M_0,g_0)$ et
d'un instanton gravitationnel $Y$ se fait en modifiant la m\'etrique $\bg$ de
$Y$ par des termes, explosant \`a l'infini, et permettant de la faire
co\"\i ncider \`a un ordre \'elev\'e avec $g_0$ apr\`es homoth\'etie. Dans \cite{Biq13},
il est montr\'e que le th\'eor\`eme \ref{th1} peut \^etre d\'eduit des formules
explicites du d\'eveloppement jusqu'\`a l'ordre 2 ($\bg+th_1+t^2h_2$). Le but
de cette section est ce calcul explicite, ainsi que la g\'en\'eralisation de la
m\'ethode \`a tous les orbifolds ALE de rang 1.

\subsection{Termes d'ordre 2 du recollement}
\label{sec:termes-dordre-2}

Les r\'esultats de \cite{Biq13}, obtenus en recollant une m\'etrique de
Eguchi-Hanson, s'\'etendent au recollement d'un espace ALE de rang 1. En effet,
la th\'eorie de d\'eformation d'un tel espace est similaire \`a celle de la
m\'etrique de Eguchi-Hanson : l'espace des d\'eformations d'Einstein, isomorphe
\`a $H^2(M,\setR)\otimes\setR^3\simeq \setR^3$, a la m\^eme dimension que celui pour les m\'etriques de
Eguchi-Hanson, et l'asymptotique de la forme harmonique $L^2$ est la m\^eme
que pour la m\'etrique de Eguchi-Hanson.

Rappelons que la lin\'earisation en jauge de Bianchi de l'\'equation
d'Einstein est $P_\bg=d_-d_-^*$ sur la partie sans trace de la
m\'etrique, et $\frac12 \Delta$ sur la partie \`a trace, voir (\ref{eq:16}). Le
noyau $L^2$ de $P_\bg$, qui est aussi son conoyau, est donc constitu\'e
des trois tenseurs
\begin{equation}
 o_i = \Omega \otimes \omega_i,\label{eq:39}
\end{equation}
o\`u $\Omega$ est la forme harmonique $L^2$ \'etudi\'ee \`a la section
pr\'ec\'edente. Un point important dans la r\'esolution est l'existence de
2-tenseurs sym\'etriques $k_i$, en jauge de Bianchi, tels que $P_\bg
k_i=o_i$ : la d\'emonstration de \cite[proposition 2.1]{Biq13} s'\'etend
en utilisant le lien entre $\Omega$ et $m$ \'etabli section
\ref{sec:la-cohomologie-l2}.

La m\'ethode utilis\'ee dans \cite{Biq13} consiste \`a modifier la m\'etrique ALE
$\bg$ par les termes d'ordre 2 de $g_0$ au point orbifold. On choisit donc
un 2-tenseur sym\'etrique homog\`ene dans $\setR^4$, de type
$$ H = H_{ijkl}x^ix^jdx^kdx^l, $$
d\'ecrivant le terme d'ordre 2 de $g_0$ en $p_0$ : $g_0=\euc+H+
O(x^4)$. Le tenseur de courbure de $g_0$ au point $p_0$ ne d\'epend que
de $H$ et sera not\'e $R(H)$. Quitte \`a faire agir un diff\'eomorphisme
$\psi$ tel que $\psi-\Id=O(x^3)$, on peut supposer que $H$ est en jauge de
Bianchi :
\begin{equation}
 B_\euc H=0.\label{eq:14}
\end{equation}
Le probl\`eme \`a r\'esoudre sur l'espace ALE $Y$ est, pour un 2-tenseur
sym\'etrique $h_1=\lambda-\sum_1^3 \omega_i\circ \phi_i$ :
\begin{equation}
  \label{eq:7}
  \begin{split}
    d_\bg\Ric(h_1) &= \Lambda\bg, \\
    B_\bg h_1 &= 0, \\
    h_1 &= H + O(r^{-2+\epsilon}) \text{ pour tout }\epsilon>0, \\
    \int_\Sigma \phi_i &= 0, \quad i=1,2,3.
  \end{split}
\end{equation}
La derni\`ere \'equation vise \`a \'ecarter l'ambigu\"\i t\'e provenant du noyau
(\ref{eq:39}). Pour la m\'etrique de Eguchi-Hanson, il s'agit exactement
de la condition $\int_\Sigma\langle h_1,o_i\rangle=0$ qui appara\^\i t dans \cite{Biq13}.

En g\'en\'eral, on ne peut pas r\'esoudre le probl\`eme (\ref{eq:7}), et on
r\'esout \`a la place, pour des r\'eels $\lambda_1$, $\lambda_2$, $\lambda_3$,
\begin{equation}
  \label{eq:17}
    \begin{split}
    d_\bg\Ric(h_1) &= \Lambda\bg + \sum_1^3 \lambda_io_i, \\
    B_\bg h_1 &= 0, \\
    h_1 &= H + O(r^{-2+\epsilon}) \text{ pour tout }\epsilon>0, \\
    \int_\Sigma \phi_i &= 0, \quad i=1,2,3.
  \end{split}
\end{equation}
Comme montr\'e dans \cite[\S{} 3]{Biq13}, si $H$ satisfait
(\ref{eq:14}), alors le probl\`eme (\ref{eq:17}) est \'equivalent au
probl\`eme
\begin{equation}
  \label{eq:15}
  \begin{split}
    P_\bg h_1 &= \Lambda\bg + \sum_1^3 \lambda_io_i, \\
    h_1 &= H + O(r^{-2+\epsilon}) \text{ pour tout }\epsilon>0, \\
    \int_\Sigma \phi_i &= 0, \quad i=1,2,3.
  \end{split}
\end{equation}

Le premier lemme suivant g\'en\'eralise au cas d'un instanton de rang 1
les r\'esultats \cite[\S{} 3-4]{Biq13} sur l'existence d'une solution au
syst\`eme (la d\'emonstration est la m\^eme). Le second donne le calcul
explicite, crucial dans la suite, des termes d'ordre 2 de la courbure
de $\bg+th_1$.

\begin{lemm}\label{lem:lambda_i}
  Le probl\`eme (\ref{eq:17}) a toujours une solution, avec les $\lambda_i$
  donn\'es par
%  $$ \lambda_i = \frac k{2(k+1)} \frac{\Vol\Sigma}{2\pi} \langle R_+(H)(I_1),I_i\rangle. $$
  $$ \lambda_i = \frac {\pi\Vol\Sigma}{\|\Omega\|^2} \langle R_+(H)(I_1),I_i\rangle. $$
  En particulier, le syst\`eme (\ref{eq:7}) a une solution si et seulement si
  $$R_+(H)(I_1)=0.$$\qed
\end{lemm}
La formule donn\'ee ici co\"\i ncide avec celle obtenue pour la singularit\'e $A_1$
dans \cite{Biq13}. Nous redonnons ici une preuve, diff\'erente, car nous
avons besoin de la valeur pr\'ecise de la constante, qui d\'epend de la singularit\'e.

Le facteur $\|\Omega\|^2$, la valeur exacte duquel nous n'avons pas besoin dans la
suite, se calcule n\'eanmoins facilement comme nombre d'intersection orbifold
: par exemple, si la singularit\'e r\'esiduelle est $A_{k-1}$, alors
\begin{equation}
  \| \Omega \|^2 = \int_Y -\Omega\land \Omega=-4\pi^2[\Sigma]^2=4\pi^2\frac{k+1}k. \label{eq:21}
\end{equation}

\begin{proof}
  Clairement,
  $$ \lambda_i = \frac1{\| o_i \|^2} \langle o_i,P_\bg h_1\rangle
         = \frac1{\| \Omega \|^2} \langle\Omega,d_-d_-^*\phi_i\rangle, $$
  et
  \begin{align*}
    \langle\Omega,d_-d_-^*\phi_i\rangle
    &= - \int_Y \Omega\land d_-d^*\phi_i \\
    &= - \int_Y \Omega\land dd^*\phi_i \\
    &= - \lim_{r\to\infty} \int_{S_r} \Omega\land d^*\phi_i \\
    &= - \lim_{r\to\infty} \int_{S_r} d^C\psi\land dd^*\phi_i \\
    &= - \lim_{r\to\infty} \int_{S_r} d^C\psi\land R_i^{(1)}\\
    &= - \lim_{r\to\infty} \int_{S_r} \langle dr\land d^C\psi,R_i^{(1)}\rangle
  \end{align*}
  puisque $R_i^{(1)}$ est autoduale. Or
  $ dr\land d^C(\frac1{r^2}) = -\frac2{r^3} dr\land Jdr $,
  donc, vu l'asymptotique de $\psi$ donn\'ee par le lemme \ref{lem:ale-potentiel},
  $$ \langle\Omega,d_-d_-^*\phi_i\rangle = \pi(\Vol \Sigma) \langle R(H)(I_i),I_1\rangle. $$
\end{proof}
\begin{lemm}\label{lem:Ric-ordre-2}
  Supposons $R_+(H)(I_1)=0$, de sorte que
  \begin{equation}
  R_+(H) = \begin{pmatrix} 0 & & \\ & R_{22} & R_{23} \\ & R_{32} & R_{33}
              \end{pmatrix}.\label{eq:8}
  \end{equation}
  Alors, au premier ordre sur $Y$, dans la direction de la d\'eformation
  $h_1=(\phi,\lambda)$ satisfaisant (\ref{eq:7}) :
  \begin{itemize}
  \item la courbure $R_+^{(1)}$ est constante, donn\'ee par (\ref{eq:8}) ;
  \item quitte \`a modifier $H$ par l'action d'un diff\'eomorphisme $\psi$ de $M_0$
  tel que $\psi-\Id=O(x^3)$, la forme de connexion sur $V$ est
  \begin{equation}
    a^{(1)}=R_{22}\alpha_2\otimes v_2+R_{23}(\alpha_2\otimes v_3+\alpha_3\otimes v_2)+R_{33}\alpha_3\otimes v_3 ,\label{eq:20}
  \end{equation}
    o\`u $\alpha_i=\frac12 J_idm$ satisfait $d\alpha_i=\omega_i$ pour $i=2$, $3$ (lemme \ref{lem:ale-s1}).
 \end{itemize}
  Enfin, les termes d'ordre $2$ de la courbure de Ricci sont donn\'es par
  \begin{multline*}
    \Ric_0^{(2)} = d_-a^{(2)} + 2(R_{22}R_{33}-R_{23}^2)(\alpha_2\land \alpha_3)_-\otimes v_1 \\-
    (R_{22}\phi_2+R_{23}\phi_3)\otimes v_2 - (R_{23}\phi_2+R_{33}\phi_3)\otimes v_3.
  \end{multline*}
\end{lemm}
\begin{proof}
  Puisque la d\'eformation infinit\'esimale est d'Einstein, par le lemme
  \ref{lem:Ric2}, les coefficients de $R_+^{(1)}$ sont harmoniques ; comme
  ils sont asymptotes \`a ceux du tenseur constant $R_+(H)$, ils sont
  n\'ecessairement constants, d'o\`u la premi\`ere assertion.

  La forme de connexion $a^{(1)}=\sum a_i^{(1)} \otimes v_i$ satisfait
  \begin{equation}\label{eq:18}
    \begin{split}
      da_i^{(1)} &= R_i^{(1)}, \\
      d^*a_i^{(1)} &= 0, \\
      a_i^{(1)} &\sim a_{i,\infty}^{(1)},
    \end{split}
  \end{equation}
  o\`u la 1-forme $a_{i,\infty}^{(1)}$, \`a coefficients lin\'eaires en les
  coordonn\'ees $x_i$, est donn\'ee par les termes d'ordre 1 en $p_0$ de la
  connexion de Levi-Civita de $g_0$.

  Il en r\'esulte que $a_i^{(1)}$ est d\'etermin\'ee \`a une 1-forme harmonique
  pr\`es, et est donc d\'etermin\'ee par son comportement asymptotique \`a
  l'infini. Pour fixer celui-ci, choisissons autour de $p_0$ des
  coordonn\'ees $(x^i)$ g\'eod\'esiques pour la m\'etrique $g_0$ (ce qui revient \`a
  l'action du diff\'eomorphisme indiqu\'e dans l'\'enonc\'e du lemme). On obtient
  \begin{equation}
  \partial_r \lrcorner a_{i,\infty}^{(1)} = 0,\label{eq:19}
  \end{equation}
  au risque d'avoir perdu la deuxi\`eme condition dans
  (\ref{eq:18}). N\'eanmoins, il s'av\`ere que ce n'est pas le cas : en effet,
  la condition (\ref{eq:19}), avec $da_i^{(1)}=R_i^{(1)}$, d\'etermine
  \'evidemment $a_{i,\infty}^{(1)}$. Or, on va voir que la forme (\ref{eq:20})
  satisfait \`a la fois le syst\`eme (\ref{eq:18}) et l'\'equation (\ref{eq:19})
  : c'est donc la solution unique cherch\'ee.

  V\'erifions maintenant que la formule propos\'ee (\ref{eq:20}) satisfait \`a la
  fois (\ref{eq:18}) et (\ref{eq:19}) : compte tenu que $d\alpha_i=\omega_i$ et
  $d^*\alpha_i=0$ pour $i=2$, $3$, l'\'equation (\ref{eq:18}) est claire ; en
  outre, pr\`es de l'infini, d'apr\`es (\ref{eq:5}),
  $$\partial_r\lrcorner \alpha_i=-\frac12 dm(J_i \partial_r)=O(r^{-3}),$$ d'o\`u
  r\'esulte (\ref{eq:19}).

  Enfin, la formule sur $\Ric_0^{(2)}$ se calcule imm\'ediatement \`a partir de
  (\ref{eq:20}) et du lemme \ref{lem:Ric2}.
\end{proof}

\subsection{Termes d'ordre 4 du recollement. Cas $A_1$, $D_k$ et $E_k$.}
\label{sec:termes-dordre-4}

L'\'etape suivante de la r\'esolution du probl\`eme d'Einstein consiste \`a faire
co\"\i ncider les termes de la modification de $\bg$ avec ceux du d\'eveloppement
de $g_0$ \`a l'ordre 4, donn\'es par un tenseur
$$ H^{(2)} = H_{ijklmn}x^ix^jx^kx^ldx^mdx^n, $$
qu'\`a nouveau on peut supposer mis en jauge de Bianchi :
$$ B_\euc H^{(2)} = 0. $$
On pose le syst\`eme, pour des r\'eels $\mu_1$, $\mu_2$, $\mu_3$ :
\begin{equation}
  \label{eq:9}
  \begin{split}
    d_\bg\Ric(h_2) &= \Lambda h_1 - \Ric^{(2)}(\bg+h_1) + \sum_1^3 \mu_i o_i, \\
    B_\bg h_2 &= 0, \\
    h_2 &= H^{(2)} + O(r^\epsilon) \text{ pour tout }\epsilon>0.   
  \end{split}
\end{equation}
L'existence d'une solution provient de \cite[lemme 14.1]{Biq13}\footnote{Il y a une faute de frappe dans le syst\`eme analogue (103) de \cite{Biq13},
o\`u le terme d'erreur est seulement en $O(r^2)$, ce qui laisse subsister une
ambigu\"\i t\'e sur les coefficients $\mu$ ; le contr\^ole $O(r^\epsilon)$ est n\'ecessaire pour
obtenir un recollement suffisamment bon de la m\'etrique $\bg+th+t^2h_2$ sur $Y$ avec $g_0$.
}, et nous sommes int\'eress\'es par le calcul de l'obstruction :
\begin{lemm}\label{lemm:calcul-mu}
  Supposons $\det R_+(p_0)=0$ (de sorte que $\lambda_1=\lambda_2=\lambda_3=0$).
  Si $\Gamma$ est le groupe correspondant \`a la singularit\'e $A_1$, $D_k$ ou
  $E_k$, alors
  $$ \mu_1 = \frac {4\pi}{\|\Omega\|^2} (R_{22} R_{33} - R_{23}^2) \int_\Sigma m\omega_1 . $$
\end{lemm}
\begin{proof}
  D'apr\`es (\ref{eq:4}), puisque $g_0$ est d'Einstein, on a
  \begin{equation}
   P_\bg H^{(2)} +\Ric^{(2)}(\bg+h_1) - \Lambda h_1 =O(r^{-2+\epsilon})\label{eq:10}
  \end{equation}
  pour tout $\epsilon>0$. On cherche alors la solution sous la forme
  $$ h_2 = \chi H^{(2)} + \eta, $$
  o\`u $\chi$ est une fonction sur $Y$ qui vaut $1$ dans un voisinage de
  l'infini, et $0$ dans une large partie compacte, et $\eta$ satisfait
  l'\'equation
  \begin{equation}
    \label{eq:11}
    P_\bg \eta - \sum_1^3 \mu_io_i = - P_\bg(\chi H^{(2)}) - \Ric^{(2)}(\bg+h_1) + \Lambda h_1.
  \end{equation}
  Par (\ref{eq:10}), le membre de droite est $O(r^{-2+\epsilon})$, d'o\`u il r\'esulte
  qu'on peut effectivement trouver une solution $h_2$ du syst\`eme
  (\ref{eq:9}), avec
  $$ \mu_1 = \frac 1{2\|\Omega\|^2} \int_Y \big\langle o_1 , P_\bg(\chi H^{(2)}) + \Ric^{(2)}(\bg+h_1) -
  \Lambda h_1 \big\rangle.$$
  Appliquant le lemme \ref{lem:Ric-ordre-2}, on obtient donc
  \begin{align*}
    \mu_1 &= \frac 1{\|\Omega\|^2} \int_Y \big\langle\Omega,d_-d^*(\chi H^{(2)})+d_-a_1^{(2)}+2(R_{22}R_{33}-R_{23}^2)(\alpha_2\land \alpha_3)_--\Lambda\phi_1\big\rangle\\
         &= -\frac 1{\|\Omega\|^2} \int_Y \Omega\land \big( d(d^*(\chi H^{(2)})+a_1^{(2)}) +
         2(R_{22}R_{33}-R_{23}^2)\alpha_2\land \alpha_3-\Lambda\phi_1 \big)\\
         &= -\frac 1{\|\Omega\|^2} \big( \lim_{r\to\infty}\int_{S_r}\Omega\land (d^*H^{(2)}+a_1^{(2)})\\
         & \phantom{= -\frac 1{\|\Omega\|^2} \big(} \quad +
         2(R_{22}R_{33}-R_{23}^2) \int_Y \Omega\land \alpha_2\land \alpha_3 - \Lambda \int_Y \Omega\land \phi_1\big). 
  \end{align*}
  Attention, dans cette \'equation, $a_1^{(2)}$ contient exclusivement les
  termes d'ordre 2 de la forme de connexion de $\bg+h_1$, ceux provenant de
  $h_2$ interviennent via $d^*H^{(2)}$.

  Examinons s\'epar\'ement les trois termes de la formule. Compte tenu de
  (\ref{eq:7}) et de (\ref{eq:6}), nous avons, dans les cas $A_1$, $D_k$ et
  $E_k$,
  \begin{equation}
  \Omega\land (d^*H^{(2)}+a_1^{(2)}) = \text{terme en }r^{-1} + O(r^{-5+\epsilon}),\label{eq:12}
  \end{equation}
  donc, int\'egr\'e contre le volume qui cro\^\i t en $r^3$, on n'obtient aucun
  terme constant quand $r\to\infty$. Le premier terme ne contribue donc pas
  \`a $\mu_1$.

  \'Evaluons maintenant la deuxi\`eme int\'egrale, en utilisant le fait que
  $\Omega=dd^C\psi$, o\`u $\psi$ est le potentiel fourni par le lemme
  \ref{lem:ale-potentiel}. On obtient imm\'ediatement, en notant $\Sigma_r$
  un voisinage tubulaire de $\Sigma$ de taille $r$,
  \begin{multline}
    \label{eq:13}
    \int_Y dd^C\psi\land \alpha_2\land \alpha_3 = \int_Y d^C\psi\land d(\alpha_2\land \alpha_3)\\
     + \lim_{r\to\infty}\int_{S_r} d^C\psi\land \alpha_2\land \alpha_3
     - \lim_{r\to0}\int_{\partial\Sigma_r}d^C\psi\land \alpha_2\land \alpha_3.
  \end{multline}
  Examinons le premier terme de (\ref{eq:13}). Pour $i=2$, $3$, on a
  $d\alpha_i=\omega_i$ et $\alpha_i=\frac12 J_idm$; avec l'identit\'e $\omega_i\land \beta=-*J_i\beta$ sur les
  1-formes, on obtient
  $$ d^C\psi\land d(\alpha_2\land \alpha_3) = - d^C\psi\land *J_1dm = - dm\land *d\psi, $$
  d'o\`u r\'esulte, puisque $\psi$ est harmonique,
  \begin{align*}
    \int_Y d^C\psi\land d(\alpha_2\land \alpha_3) &= - \int_Y dm\land *d\psi \\
    &= - \lim_{r\to\infty} \int_{S_r} m *d\psi + \lim_{r\to0} \int_{\partial\Sigma_r} m *d\psi.
  \end{align*}
  Vu l'asymptotique (\ref{eq:5}) pour $m$, le terme constant dans l'int\'egrale
  $$ \int_{S_r} d^C\psi \land \alpha_2 \land \alpha_3 - m *d\psi $$
  est nul, donc il n'y a pas de contribution venant de l'infini. Il
  reste ainsi
  \begin{align*}
    \int_Y dd^C\psi \land \alpha_2 \land \alpha_3
    &= \lim_{r\to0}\int_{\partial\Sigma_r} -d^C\psi\land \alpha_2\land \alpha_3 + m *d\psi \\
    &= 2\pi \int_\Sigma \alpha_2\land \alpha_3-m \omega_1 .
  \end{align*}
  Comme $\alpha_i=\frac12 J_idm$ avec $dm=-\xi\lrcorner \omega_1$, et comme
  $\Sigma$ est globalement pr\'eserv\'ee par l'action de cercle, le vecteur
  $\xi$ est tangent \`a $\Sigma$ et il en r\'esulte
  $\iota_\Sigma^*\alpha_2=\iota_\Sigma^*\alpha_3=0$. Finalement,
  $$ \int_Y dd^C\psi \land \alpha_2 \land \alpha_3 = -2\pi \int_\Sigma m\omega_1. $$

  La formule annonc\'ee sur $\mu_1$ se d\'eduit alors du fait que la
  troisi\`eme int\'egrale n'apporte aucune contribution : en effet, vu la
  formule (\ref{eq:20}) pour la connexion, on a $d\phi_1=0$ et donc
  \begin{align}
    \int_Y \Omega\land \phi_1 &= - \int_Y dd^C\psi\land \phi_1\notag \\
    &= \lim_{r\to\infty}\int_{S_r} d^C\psi\land \phi_1 - \lim_{r\to0}\int_{\Sigma_r} d^C\psi\land \phi_1 \label{eq:40}\\
    &= -2\pi \int_\Sigma \phi_1 = 0.\notag
  \end{align}
  Vu les d\'eveloppements de $\psi$ et $h_1$ \`a l'infini, il n'y a en effet
  pas de contribution quand $r\to\infty$.
\end{proof}

\subsection{Termes d'ordre 4 du recollement. Cas $A_k$.}
\label{sec:termes-4-Ak}

Pour traiter le cas $A_k$, nous avons besoin de quelques pr\'eliminaires sur
la structure des termes d'ordre 2 de la courbure d'une m\'etrique
d'Einstein. Toujours d\'eformant la m\'etrique hyperk\"ahl\'erienne $\bg$, \'ecrivons le
d\'eveloppement de la courbure de $\Omega^+$ comme
$R=R^{(1)}+R^{(2)}+\cdots$. L'identit\'e de Bianchi $dR+[a\land R]=0$ donne les
identit\'es :
\begin{align}
  dR^{(1)}&=0 \label{eq:24} \\
  dR^{(2)}+[a^{(1)}\land R^{(1)}]&=0. \label{eq:22}
\end{align}
Par le lemme \ref{lem:Ric-ordre-2}, et sous ses hypoth\`eses, on calcule
\begin{align*}
  [a^{(1)}\land R^{(1)}]_1
 &= 2(R_{22}R_{33}-R_{23}^2)(\alpha_2\land \omega_3-\omega_2\land \alpha_3) \\
 &= -2(R_{22}R_{33}-R_{23}^2)d(m\omega_1)
\end{align*}
car $*(\alpha_2\land \omega_3-\omega_2\land \alpha_3)=J_3\alpha_2-J_2\alpha_3=-J_1dm=-*d(m\omega_1)$. Donc, \`a
partir de (\ref{eq:22}),
\begin{equation}
  \label{eq:25}
  R^{(2)}_1= \varrho + 2(R_{22}R_{33}-R_{23}^2)m\omega_1, \quad
  d \varrho =0.
\end{equation}
On notera que, puisque $R^{(1)}_1=0$, les parties autoduales et
antiautoduales de $R^{(2)}_1$ pour la m\'etrique $\bg$ ou sa d\'eformation
co\"\i ncident. Par la condition d'Einstein, on a donc $R^{(2)}_1\in \Omega^+$
(pour la m\'etrique $\bg$) et ainsi $\varrho \in \Omega^+$ est harmonique.

Le calcul ci-avant s'applique aussi au d\'eveloppement formel de la
m\'etrique $g_0$ au point $p_0$ : dans ce cas, on se trouve sur $\setR^4$ et $m$
prend la valeur $\frac{r^2}2$, donc
\begin{equation}
  \label{eq:23}
  R^{(2)}_1 = \varpi + (R_{22}R_{33}-R_{23}^2) r^2\omega_1,
\end{equation}
o\`u les coefficients de $\varpi \in \Omega^+(\setR^4)$ sont des formes quadratiques, et
$d\varpi =0$.

Nous aurons besoin du calcul suivant.
\begin{lemm}
  Soit $\varpi=\sum_1^3 z_i\omega_i \in \Omega^+(\setR^4)$ une forme autoduale ferm\'ee, dont
  les coefficients $z_i$ sont des formes quadratiques. Soit
  $F=\frac{r_1^2-r_2^2}{r^6}$. Alors
  \begin{equation}
  \int_{S^3} d^CF\land \varpi = \frac{\pi^2}2 (-\partial^2_{11}-\partial^2_{22}+\partial^2_{33}+\partial^2_{44})z_1 .\label{eq:26}
  \end{equation}
\end{lemm}
On remarquera que le membre de gauche de (\ref{eq:26}) est homog\`ene de
degr\'e 0, ce qui donne un sens \`a la formule.
\begin{proof}
  Par un calcul direct,
  \begin{equation*}
%    d^C F &= \frac 4{r^8} \big(
%    (-r_1^2+2r_2^2)(x^1dx^2-x^2dx^1)+(-2r_1^2+r_2^2)(x^3dx^4-x^4dx^3) \big) \\
    \big(d(r^2)\land d^CF\big)_+ = \frac 4{r^6} \big(
    (r_2^2-r_1^2)\omega_1+(x^1x^4+x^2x^3)\omega_2-(x^1x^3-x^2x^4)\omega_3 \big).
  \end{equation*} 
  On d\'eduit :
  \begin{align*}
    \int_{S^3} d^CF\land \varpi
  &= \int_{S^3} \langle dr\land d^CF, \varpi \rangle \\
  &= \int_{S^3} \langle (dr\land d^CF)_+, \varpi \rangle \\
  &= 4 \int_{S^3} \frac18 \frac{(r_2^2-r_1^2)^2}{r^7}
  (-\partial^2_{11}-\partial^2_{22}+\partial^2_{33}+\partial^2_{44})z_1 \\
  & \qquad + \frac12 \frac{(x^1x^4+x^2x^3)^2}{r^7} (\partial^2_{14}+\partial^2_{23})z_2\\
  & \qquad - \frac12 \frac{(x^1x^3-x^2x^4)^2}{r^7} (\partial^2_{13}-\partial^2_{24})z_3 
  \end{align*}
  Si l'on suppose en outre $d\varpi=0$, alors on voit ais\'ement que
  $$ (\partial^2_{14}+\partial^2_{23})z_2 - (\partial^2_{13}-\partial^2_{24})z_3
    = \frac12(-\partial^2_{11}-\partial^2_{22}+\partial^2_{33}+\partial^2_{44})z_1. $$
  Comme par ailleurs
  $$ \int_{S^3} (x^1x^4+x^2x^3)^2 = \frac14 \int_{S^3} (r_1^2-r_2^2)^2 =
    \frac{\pi^2}6,$$
  la formule devient
  $$ \int_{S^3} d^CF\land \varpi = \frac{\pi^2}2
  (-\partial^2_{11}-\partial^2_{22}+\partial^2_{33}+\partial^2_{44})z_1.$$
\end{proof}
\begin{rema}\label{rem:DFvarpi}
  Si on fait le m\^eme calcul avec $\varpi \in \Omega^-(\setR^4)$, avec les autres
  hypoth\`eses inchang\'ees, on trouve $\int_{S^3}d^CF\land \varpi =0$. (Ce fait peut
  s'expliquer aussi en voyant que $d^CF$ et $\varpi$ vivent dans des repr\'esentations
  irr\'eductibles distinctes de $SO_4$).
\end{rema}

Nous pouvons maintenant passer au calcul de $\mu_1$ dans les cas
restants :
\begin{lemm}
  Dans le cas $A_k$, on a
  \begin{multline*}
    \mu_1 = \frac{(\Vol\Sigma)^2}{\|\Omega\|^2}  \Big\{
           (k+1) (R_{22}R_{33}-R_{23}^2) \\
         - \frac1{16} (k-1) \big\langle(\nabla^2_{11}+\nabla^2_{22}-\nabla^2_{33}-\nabla^2_{44})R(p_0)(I_1),I_1\big\rangle \Big\}.
  \end{multline*}
\end{lemm}
\begin{proof}
  Le calcul effectu\'e dans le lemme \ref{lemm:calcul-mu} reste valable,
  \`a l'exception de termes suppl\'ementaires dans les \'equations
  (\ref{eq:12}) et (\ref{eq:40}), qui apportent une autre contribution
  \`a $\mu_1$. En fait, par la remarque \ref{rem:DFvarpi}, la contribution
  apport\'ee par (\ref{eq:40}) est nulle, et il ne reste \`a analyser que
  celle provenant de (\ref{eq:12}) : \'ecrivant $\Omega=dd^C\psi$, on obtient
  $$ \int_{S_r} \Omega\land (d^*H^{(2)}+a_1^{(2)})
  = \int_{S_r} d^C\psi \land d(d^*H^{(2)}+a_1^{(2)}) = \int_{S_r} d^C\psi \land
  R_1^{(2)}. $$ Les termes de $R_1^{(2)}$ donnant une contribution \`a
  la limite se r\'eduisent \`a ceux de $g_0$ en $p_0$, donn\'es par
  (\ref{eq:23}). Vu la forme explicite de $d^C\psi$, il est clair que le
  terme en $r^2\omega_1$ de $R_1^{(2)}$ ne contribue pas. Vu l'asymptotique
  de $\psi$ donn\'ee par le lemme \ref{lem:ale-potentiel}, la formule
  (\ref{eq:26}) fournit une contribution \`a $\mu_1$ donn\'ee par
  \begin{multline*}
    -\frac1{\|\Omega\|^2} \lim_{r\to\infty}\int_{S_r} \Omega\land (d^*H^{(2)}+a_1^{(2)})\\
    = - \frac{(k-1)(\Vol\Sigma)^2}{16\|\Omega\|^2} 
    \langle(\nabla^2_{11}+\nabla^2_{22}-\nabla^2_{33}-\nabla^2_{44})R(I_1),I_1\rangle .
  \end{multline*}
  (On obtient les d\'eriv\'ees covariantes car, au point $p_0$, on a
  $R(I_1)=0$). En ajoutant la contribution d\'ej\`a calcul\'ee dans le lemme
  \ref{lemm:calcul-mu}, et compte tenu de (voir annexe \ref{sec:les-instantons-a_k})
  $$ \int_\Sigma m\omega_1 = \pi(k+1)\big( \frac{\Vol\Sigma}{2\pi} \big)^2,$$
  on trouve la valeur indiqu\'ee dans l'\'enonc\'e.
\end{proof}

\section{R\'esolution du probl\`eme d'Einstein}
\label{sec:resol-du-probl}

Dans cette section, nous commen\c cons par rappeler les techniques de
\cite{Biq13} pour r\'esoudre l'\'equation d'Einstein. Nous appliquerons ensuite
les calculs explicites vus dans la section pr\'ec\'edente pour obtenir les
r\'esultats souhait\'es.

Les deux r\'esultats suivants \'etendent les r\'esultats de \cite{Biq13} au cas du
recollement d'un instanton gravitationnel orbifold de rang 1 (au lieu de la
m\'etrique de Eguchi-Hanson). La d\'emonstration est identique.
\begin{theo}[{\cite[\S{} 14]{Biq13}}]\label{theo:resol-avec-obstruction}
  Soit $(M_0,g_0)$ asymptotiquement hyperbolique, Einstein, non d\'eg\'en\'er\'e,
  avec un point orbifold $p_0$ de type $\setC^2/\Gamma$, et $Y$ un instanton
  gravitationnel orbifold de rang 1, asymptotique \`a $\setC^2/\Gamma$. Soit $M$ le
  recollement (topologique) de $M_0$ et $Y$. Alors pour $t$ assez petit, on
  peut r\'esoudre l'\'equation
  $$ \Ric(g_t) - \Lambda g_t = \sum_1^3 \lambda_i(t) o_{i,t} , $$
  o\`u $g_t$ est une m\'etrique asymptotiquement hyperbolique sur $M$, 
  et on a le d\'eveloppement
  $$ \lambda_i(t) = t \lambda_i + t^2 \mu_i + O(t^{\frac52}), $$
  o\`u les $\mu_i$ sont les constantes obtenues dans la r\'esolution du probl\`eme
  (\ref{eq:9}).

  Toute la construction d\'epend d'une mani\`ere lisse :
  \begin{itemize}
  \item de la m\'etrique $g_0$, et en particulier de son infini conforme
    $\gamma$ ;
  \item d'un param\`etre de recollement en $p_0$, $\varphi\in Sp_1$.
  \end{itemize}
\qed
\end{theo}
Quand varient le param\`etre $\varphi$ ou l'infini conforme $\gamma$, notons
explicitement la d\'ependance en $\gamma$ : $g_0(\varphi,\gamma)$, $\lambda_i(t,\varphi,\gamma)$, $\lambda_i(\varphi,\gamma)$,\ldots
On a :
\begin{lemm}[{\cite[\S{} 12]{Biq13}}]
  1\textdegree{} Il existe sur $\partial M_0$ des 2-tenseurs sym\'etriques sans trace, $\dot\gamma_1$,
  $\dot\gamma_2$, $\dot\gamma_3$, tels que
  $$ \frac{\partial\lambda_i}{\partial\gamma}(\dot \gamma_j) = \delta_{ij}. $$

  2\textdegree{} Si $R_+(p_0)$ est de rang 2 ($R_{22}R_{33}-R_{23}^2\neq 0$), alors
  pour $i=2$, $3$, on a
  $$ \frac{\partial\lambda_i}{\partial\varphi}\neq 0. $$\qed
\end{lemm}

Des deux r\'esultats pr\'ec\'edents, on d\'eduit : 
\begin{theo}\label{th:detRpositif}
  On se place sous les hypoth\`eses du th\'eor\`eme
  \ref{theo:resol-avec-obstruction}, en supposant $\det
  \bR_+(p_0)=0$. Quitte \`a faire agir un \'el\'ement de $Sp_1$, on peut
  supposer que $\bR_+(p_0)(I_1)=0$. Supposons de surcro\^\i t l'annulation non
  d\'eg\'en\'er\'ee ($R_{22}R_{33}-R_{23}^2\neq 0$). Alors :

  1\textdegree{} Pour $t$ petit, il existe $\varphi(t)\in SO_4$ et une fonction $z(t)$ tels que
  $$ (\Ric-\Lambda)\big(g_t(\varphi(t),\gamma_0+z(t)\dot \gamma_1)\big) = 0, $$
  et en outre, $z(t)=-\mu_1t + O(t^{\frac32})$.
  
  2\textdegree{} Dans l'espace $\cC$ des m\'etriques conformes sur $\partial M$, l'hypersurface
  $$ \cC_0 = \{ \gamma, \det \bR_+(g_0(\gamma))(p_0)=0 \} $$
  est lisse. Pour les singularit\'es de type $A_1$, $D_k$ et $E_k$, les
  infinis conformes des m\'etriques d'Einstein d\'esingularis\'ees sont du c\^ot\'e
  de $\cC_0\subset \cC$ d\'etermin\'e par
  \begin{equation}
   \det \bR_+\big(g_0(\gamma)\big)(p_0) > 0.\label{eq:32}
 \end{equation}
\end{theo}
\begin{rema}
  Si l'annulation est d\'eg\'en\'er\'ee, on peut toujours d\'esingulariser
  $g_0$, mais en utilisant, au lieu d'un \'el\'ement de $Sp_1$, les
  d\'eformations de l'infini conforme dans les directions $\dot \gamma_2$ et
  $\dot \gamma_3$ pour compenser $\lambda_2(t)$ et $\lambda_3(t)$.
\end{rema}
\begin{proof}
  La premi\`ere partie d\'ecoule de \cite[\S{} 14]{Biq13}, la lissit\'e de $\cC_0$
  de $\frac{\partial\lambda_1}{\partial\gamma}\neq 0$. Reste \`a d\'eterminer de quel c\^ot\'e de $\cC_0$ sont
  obtenues les d\'esingularisations : l'infini conforme de $g_t$ est
  $$ \gamma_t=\gamma_0-\mu_1t\dot\gamma_1+o(t),$$
  d'o\`u il r\'esulte 
  $$ \frac d{dt} \lambda_1(\gamma_t) \big|_{t=0}
    = - \mu_1 \frac{\partial\lambda_1}{\partial t} (\dot\gamma_1) = -a (R_{22}R_{33}-R_{23}^2) $$
  avec $a>0$ d'apr\`es le lemme \ref{lemm:calcul-mu}. Finalement,
  $$ \frac d{dt} \det R_+\big(g_0(\gamma_t)\big)(p_0) \big|_{t=0}
    = -a (R_{22}R_{33}-R_{23}^2)^2 < 0 . $$
  Pour $\bR_+=-R_+$, on obtient le signe oppos\'e.
\end{proof}
\begin{rema}
  Dans le cas $A_k$ pour $k>1$, vu la forme de $\mu_1$, on ne peut pas
  dire a priori de quel c\^ot\'e de l'hypersurface se trouveront les
  m\'etriques (partiellement) d\'esingularis\'ees. N\'eanmoins, d'autres
  obstructions doivent s'annuler pour d\'esingulariser compl\`etement
  $M_0$ : l'obstruction suivante, calcul\'ee au lemme \ref{lem:A-Ak},
  donne la valeur
  $\langle(\nabla^2_{11}+\cdots-\nabla^2_{44})R(p_0)I_1,I_1\rangle=16\frac{k-1}{k+1}(R_{22}R_{33}-R_{23}^2)$,
  qui conduit \`a
  \begin{align*}
    \mu_1&=\frac{4k(\Vol\Sigma)^2}{(k+1)\|\Omega\|^2}(R_{22}R_{33}-R_{23}^2)\\
       &=a(R_{22}R_{33}-R_{23}^2)
  \end{align*}
  avec $a>0$, si bien que les d\'esingularisations restent bien du c\^ot\'e
  d\'etermin\'e par (\ref{eq:32}).
\end{rema}

\section{D\'esingularisation des singularit\'es de rang plus \'elev\'e}
\label{sec:desing-des-sing}

Dans cette section, nous amor\c cons l'\'etude de la d\'esingularisation des
singularit\'es de rang $k>1$. Nous envisageons les aspects formels, qui
fournissent l'obstruction venant apr\`es la premi\`ere obstruction $\det
R_+(p_0)=0$, et expliquent pourquoi le th\'eor\`eme \ref{th:detRpositif}
s'\'etend d'une certaine mani\`ere aux singularit\'es $A_k$.

\subsection{Cas $A_k$}
Limitons-nous pour le moment \`a une singularit\'e initiale $A_k$.
Repla\c cons-nous sous les hypoth\`eses du th\'eor\`eme \ref{th:detRpositif} :
$(M_0,g_0)$ est asymptotiquement hyperbolique, Einstein, non
d\'eg\'en\'er\'ee, avec un point orbifold $p_0$ de rang $k$, et on recolle un
instanton gravitationnel $Y$, de rang 1. Ici nous prenons $Y$ avec un seul
point orbifold $p_1$ (donc singularit\'e de type $A_{k-1}$).

Par le th\'eor\`eme \ref{th:detRpositif}, il existe une m\'etrique d'Einstein
$g_t$ sur le recollement topologique $M$, d'infini
conforme
$$ \gamma_t = - t \mu_1 \dot \gamma_1 + O(t^{\frac32}). $$
Il est naturel de poursuivre la d\'esingularisation au point
$p_1$. R\'eappliquant le th\'eor\`eme \ref{sec:resol-du-probl}, on est men\'e \`a
examiner le comportement de $\det R_+(g_t)(p_1)$. Pour cela, observons que
la variation au premier ordre de $R_+$ sur $Y$ a \'et\'e calcul\'ee dans le lemme
\ref{lem:Ric-ordre-2}, de sorte que
\begin{equation}
 R_+(p_1) = 
\begin{pmatrix}
  0 & & \\ & tR_{22} & tR_{23} \\ & tR_{32} & tR_{33}
\end{pmatrix}
+ O(t^2).\label{eq:33}
\end{equation}
Notons le coefficient en haut \`a gauche
\begin{equation}
 R_{11}(p_1) = A t^2 + O(t^3).\label{eq:34}
\end{equation}
Alors, clairement,
\begin{equation}
  \label{eq:27}
  \det R_+(g_t)(p_1) = (R_{22}R_{33}-R_{23}^2) A t^4 + O(t^5).
\end{equation}

\begin{lemm}\label{lem:A-Ak}
  Sous les hypoth\`eses pr\'ec\'edentes (singularit\'e de type $A_k$, et $Y$ a un
  seul point singulier), le coefficient $A$ est donn\'e par :
  \begin{multline*}
    A = \frac{\Vol\Sigma}{2\pi} \big( -(k-1) (R_{22}R_{33}-R_{23}^2) \\ +
    \frac1{16}(k+1)\langle(\nabla^2_{11}+\nabla^2_{22}-\nabla^2_{33}-\nabla^2_{44})R(p_0)I_1,I_1\rangle
    \big).
  \end{multline*}
\end{lemm}
\begin{rema}
  Comme vu dans la remarque \ref{rema:11223344}, on peut utiliser la
  transformation $(z^1,z^2)\mapsto(-z^2,z^1)$ pour modifier le signe du
  second terme, donc il y a deux valeurs de
  $\langle(\nabla^2_{11}+\nabla^2_{22}-\nabla^2_{33}-\nabla^2_{44})R(p_0)I_1,I_1\rangle$ qui annulent
  $A$.
\end{rema}
\begin{proof} Les termes d'ordre 2 de $R_{11}$ proviennent, d'une part
  des termes d'ordre 2 de la courbure calcul\'es dans (\ref{eq:25}),
  d'autre part de la correction $\gamma_t$ \`a l'ordre 1 de l'infini
  conforme, qui provoque une modification \`a l'ordre 2 de la m\'etrique,
  et donc de la courbure, sur $Y$.
  
  D'apr\`es (\ref{eq:25}), les termes d'ordre 2 de la courbure dans
  $R_{11}$ sont
  $$ \frac12 \langle\varrho,\omega_1\rangle(p_1) + 2(R_{22}R_{33}-R_{23}^2)m(p_1). $$
  \`A ces termes il convient d'ajouter la contribution provenant de la
  modification de l'infini conforme : $\lambda_1$ est modifi\'ee de sorte que
  $\frac{d\lambda_1}{dt}=-\mu_1$. Notons $\lambda_1=\frac12 \langle R_+(H)(I_1),I_1\rangle \lambda'_1$
  d'apr\`es le lemme \ref{lem:lambda_i}, alors $R_{11}$ est
  modifi\'e (\`a l'ordre 2) par un terme global constant $-\frac{\mu_1}{\lambda'_1}$.
  Finalement, on obtient la formule
  \begin{equation}
    \label{eq:30}
    A = \frac12 \langle\varrho,\omega_1\rangle(p_1) + 2(R_{22}R_{33}-R_{23}^2)m(p_1) -\frac{\mu_1}{\lambda'_1}.
  \end{equation}

  Calculons le premier terme. Ici la forme $\varrho\in \Omega^+$ est harmonique et
  ses coefficients \`a l'infini sont des formes quadratiques invariantes
  sous $\setZ_{k+1}$. (Pour $D_k$ et $E_k$ il n'y a pas de telle forme et
  donc pas de contribution).  Comme vu dans la d\'emonstration du lemme
  \ref{lem:ale-potentiel}, il y a trois formes quadratiques
  invariantes \`a l'infini, $q_1=|z^1|^2-|z^2|^2$, $q_2+iq_3=z^1z^2$,
  qui se prolongent en des fonctions harmoniques $\varphi_i$ sur $Y$. On a
  vu que $\varphi_2+i\varphi_3$ est la fonction holomorphe globale $z^1z^2$,
  donc s'annule en $p_1$ ; il ne reste donc que $\varphi_1$ et on a
  $\varphi_1(p_1) = 2 \frac{\Vol\Sigma}{2\pi}$ par (\ref{eq:28}). Par cons\'equent,
  si, \`a l'infini,
  $$ \langle\varrho,I_1\rangle \sim \sum_1^3 a_iq_i, $$
  alors
  $$ \langle\varrho,I_1\rangle(p_1) = 2a_1\frac{\Vol\Sigma}{2\pi}.$$
  Comme $a_1=\frac18
  \langle(\nabla^2_{11}+\nabla^2_{22}-\nabla^2_{33}-\nabla^2_{44})R(p_0)I_1,I_1\rangle$, on obtient
  une contribution
  \begin{equation}
    \label{eq:29}
    \frac12 \langle\varrho,\omega_1\rangle(p_1) = \frac18 \frac{\Vol\Sigma}{2\pi} 
    \langle(\nabla^2_{11}+\nabla^2_{22}-\nabla^2_{33}-\nabla^2_{44})R(p_0)I_1,I_1\rangle.
  \end{equation}

  Rappelons la formule
  \begin{multline*}
    \mu_1=\frac1{\|\Omega\|^2} \big\{
    4\pi (R_{22}R_{33}-R_{23}^2) \int_\Sigma m\omega_1 \\ -
    \frac1{16}(k-1)(\Vol\Sigma)^2
    \langle(\nabla^2_{11}+\nabla^2_{22}-\nabla^2_{33}-\nabla^2_{44})R(p_0)I_1,I_1\rangle \big\},
  \end{multline*}
  o\`u le second terme n'existe que pour la singularit\'e $A_k$. Il en
  r\'esulte
  \begin{multline}\label{eq:31}
    A = 2(R_{22}R_{33}-R_{23}^2)\big(m(p_1)-\frac1{\Vol\Sigma}\int_\Sigma
    m\omega_1\big)\\
    +\frac1{16} (k+1) \frac{\Vol\Sigma}{2\pi} 
    \langle(\nabla^2_{11}+\nabla^2_{22}-\nabla^2_{33}-\nabla^2_{44})R(p_0)I_1,I_1\rangle.
  \end{multline}
  Les calculs explicites faits \`a l'annexe \ref{sec:les-instantons-a_k}
  donnent $m(p_1)=\frac{\Vol\Sigma}{2\pi}$, d'o\`u vient le r\'esultat annonc\'e.
\end{proof}

Poursuivons avec quelques sp\'eculations sur la d\'esingularisation de
$M_0$. On peut montrer qu'existe sur $\partial M_0$ une direction de
d\'eformation de la m\'etrique conforme, $\delta_1$, qui modifie la m\'etrique
d'Einstein $g_0(\gamma_1,\delta_1)$ au point $p_0$ en gardant inchang\'es les
termes d'ordre 2, mais en modifiant les termes d'ordre 4 de sorte que
$\frac \partial{\partial\delta_1} \langle(\nabla^2_{11}+\nabla^2_{22}-\nabla^2_{33}-\nabla^2_{44})R(p_0)I_1,I_1\rangle \neq
0$. Apr\`es homoth\'etie par un facteur $t$, une variation $t\delta_1$ de $\gamma$
modifie la m\'etrique $h_2$ par des termes d'ordre $t^3$, et en
particulier modifie le terme $R_{11}(p_1)$ dans (\ref{eq:33}) par un
terme $at^3$ avec $a\neq 0$, et donc $\det R_+(g_t)(p_1)$ par un terme
$$ a(R_{22}R_{33}-R_{23}^2)t^5, \quad a\neq 0.$$
Comparant avec (\ref{eq:27}), on voit ainsi que si $A=0$, les termes
suivants dans $\det R_+(g_t)(p_1)$ peuvent \^etre compens\'es par une variation
de l'infini conforme dans la direction $\delta_1$.

Synth\'etisons cette discussion : 
\begin{itemize}
\item $A=0$ (et donc la quantit\'e calcul\'ee dans le lemme \ref{lem:A-Ak})
  appara\^\i t comme obstruction d'ordre sup\'erieur \`a poursuivre la
  d\'esingularisation au point $p_1$ ;
\item quitte \`a modifier l'infini conforme dans les directions $\gamma_1$ et
  $\delta_1$, on peut r\'esoudre le probl\`eme d'Einstein et la condition de
  d\'eterminant nul au point singulier :
  $$ (\Ric-\Lambda)(g_t)=0, \qquad \det R_+(g_t)(p_1)=0. $$
\end{itemize}
Malheureusement, \`a ce stade, on ne peut pas conclure en appliquant
directement le th\'eor\`eme de d\'esingularisation au point $p_1$, car il n'y a
aucune raison que la m\'etrique $g_t$ satisfasse l'hypoth\`ese de non
d\'eg\'en\'erescence (m\^eme si on peut s'attendre \`a ce qu'elle soit v\'erifi\'ee
g\'en\'eriquement). Il est donc n\'ecessaire de n'utiliser que la non
d\'eg\'en\'erescence de $M_0$, ce qui requiert une analyse plus subtile qui sera
abord\'ee dans un autre article.

\subsection{Cas $D_k$ et $E_k$}
Tout le calcul fait dans la d\'emonstration du lemme \ref{lem:A-Ak} s'\'etend
au cas $D_k$ et $E_k$, \`a la diff\'erence que les termes impliquant
$\nabla^2R(p_0)$ disparaissent. La formule (\ref{eq:31}) a une signification
g\'eom\'etrique claire : si l'action de $\setC^*$ sur $\Sigma$ est triviale, alors $m$
est constante sur $\Sigma$ (puisque $dm=-\xi\lrcorner \omega_1$) et donc
$A=0$. C'est le cas pour $D_k$ et $E_k$ si on choisit correctement
l'instanton gravitationnel $Y$ : la d\'esingularisation contient une
configuration de courbes correspondant au diagramme de Dynkin ; si on
contracte toutes les courbes sauf l'unique courbe qui en coupe trois
autres, on obtient un instanton $Y$ avec trois points singuliers de type
orbifold sur $\Sigma$. L'action de $\setC^*$, fixant ces trois points, doit \^etre
triviale, d'o\`u r\'esulte :
\begin{lemm}
  Dans les cas $D_k$ et $E_k$, si la courbe $\Sigma$ de $Y$ correspond au n\oe ud
  du diagramme de Dynkin duquel partent trois ar\^etes, alors $A=0$.\qed
\end{lemm}
Ce lemme donne l'espoir d'un proc\'ed\'e de d\'esingularisation des singularit\'es
$D_k$ et $E_k$ en d\'esingularisant alors les trois points singuliers
r\'esiduels, tous de type $A_\ell$. Comme on pouvait s'y attendre au vu des
termes possibles invariants sous $D_k$ ou $E_k$, aucune obstruction ne
provient des termes d'ordre 4 de la m\'etrique en $p_0$, mais l'obstruction
suivante appara\^\i t probablement sur les termes d'ordre 6, contenant donc des
termes en $\nabla^4R(p_0)$.

\appendix

\section{Les instantons $A_k$}
\label{sec:les-instantons-a_k}

Dans cette annexe, nous donnons rapidement les formules explicites utilis\'ees
sur les instantons gravitationnels de type $A_k$. Commen\c cons par rappeler
l'ansatz de Gibbons-Hawking.

Soit $V$ une fonction harmonique sur $\setR^3$, telle que $*dV=d\eta$, o\`u $\eta$ est
la 1-forme de connexion d'un fibr\'e en cercles (ce qui exige que la classe
de cohomologie de la forme ferm\'ee $\frac{*dV}{2\pi}$ soit enti\`ere). Alors,
sur l'espace total du fibr\'e, la formule
$$ g=V\big((dx^1)^2+(dx^2)^2+(dx^3)^2\big)+V^{-1}\eta^2 $$
d\'efinit une m\'etrique hyperk\"al\'erienne, pour laquelle $J_1dx^1=V^{-1}\eta$ et
$J_1dx^2=dx^3$, et $J_2$ et $J_3$ sont d\'efinies de mani\`ere similaire par
permutation circulaire sur $(x^1,x^2,x^3)$.

Les instantons gravitationnels de type $A_k$ sont obtenus en prenant
$$ V = \frac12 \sum_0^k \frac 1{|x-p_i|} $$
sur $\setR^3-\{p_i\}$, o\`u les $p_i$ sont des points distincts de $\setR^3$. Compte
tenu de l'invariance par translation, on peut supposer
$$ \sum_0^k p_i = 0. $$
En ajoutant un point au-dessus de chaque $p_i$, on peut montrer que la
m\'etrique s'\'etend en une m\'etrique lisse. Quand certains $p_i$ sont
confondus, on obtient une singularit\'e orbifold. Le cas de rang 1 avec une
seule singularit\'e orbifold restante correspond \`a $k$ points confondus,
$p_1=\cdots =p_k$. Vu l'action de $SO_3$ sur la situation, on peut supposer ces
points situ\'es sur l'axe des $x^1$, et donc
$$ p_0 = (-k\lambda,0,0), \quad p_1=\cdots =p_k=(\lambda,0,0). $$
Alors on a une singularit\'e $A_{k-1}$ au point $p_1$, et la courbe
$J_1$-holomorphe $\Sigma$ se trouve au-dessus du segment $[p_0,p_1]$. En
particulier,
$$ \Vol\Sigma = \int_\Sigma dx^1\land \eta = 2\pi(k+1)\lambda. $$

Notons $\rho$ le rayon dans $\setR^3$ et $r$ le rayon dans $\setR^4$. Vu que $V\sim
\frac{k+1}{2\rho}$, on obtient \`a l'infini, en prenant $r^2=2(k+1)\rho$,
\begin{align*}
  g &\sim \frac{k+1}{2\rho}(d\rho^2+\rho^2g_{S^2}) + \frac{2r}{k+1}\eta^2 \\
    &\sim  dr^2+\frac{r^2}4+\frac{r^2}{(k+1)^2}\eta^2 ,
\end{align*}
qui est la m\'etrique standard de $\setR^4/\setZ_{k+1}$.

La situation est invariante par les rotations en les coordonn\'ees
$(x^2,x^3)$, qui donne l'action de cercle cherch\'ee (il faut prendre une
action de poids 2, par $e^{2i\theta}(x^2+ix^3)$, car les $x^i$ sont quadratiques
en les coordonn\'ees de $\setR^4$). L'application moment $m$ est
$$ m = k |x-p_1| + |x-p_0|, $$
et $\frac m2$ est un potentiel pour $\omega_2$ et $\omega_3$. On observera que, comme
il se doit, $m\sim (k+1)\rho \sim \frac{r^2}2$. En particulier, on d\'eduit la formule
utilis\'ee dans la d\'emonstration du lemme ,
\begin{align*}
  \int_\Sigma m\omega_1 &= 2\pi \int_{-k\lambda}^\lambda \big((x^1+k\lambda)+k(\lambda-x^1)\big)dx^1 \\ &= \pi(k+1)^3\lambda^2 =
  \pi(k+1)\big(\frac{\Vol\Sigma}{2\pi}\big)^2 .
\end{align*}

Enfin, une des vertus de l'ansatz de Gibbons-Hawking est de fournir les
fonctions harmoniques $x^i$. On calcule facilement le d\'eveloppement
$$ x^1 \sim \frac{|z^1|^2-|z^2|^2}{2(k+1)}. $$
Il en r\'esulte que la fonction harmonique $\varphi_1$ dans la d\'emonstration du
lemme \ref{lem:ale-potentiel} n'est autre que
\begin{equation}
\varphi_1 = 2(k+1)x^1.\label{eq:28}
\end{equation}
En particulier, le calcul fait dans ce lemme est justifi\'e par la formule
\begin{align*}
  \int_\Sigma \varphi_1 &= 2(k+1) 2\pi\int_{-k\lambda}^\lambda x^1dx^1 \\ &= - 2\pi(k+1)^2(k-1)\lambda^2 =-2\pi(k-1)
  \big(\frac{\Vol\Sigma}{2\pi}\big)^2 .
\end{align*}

\bibliography{biblio,biquard,obs} \bibliographystyle{smfplain}

\end{document}